\documentclass[12pt,a4paper]{article}
\usepackage{latexsym,amssymb,amsmath,amsthm}
\usepackage{enumitem}
\usepackage{xcolor}
\usepackage{graphicx}
\usepackage{subfig}

\newtheorem{theorem}{Theorem}[section]
\newtheorem{lemma}[theorem]{Lemma}

\newtheorem{proposition}[theorem]{Proposition}
\newtheorem{observation}[theorem]{Observation}

\newtheorem{claim}{Claim}

\newcommand{\claimproofend}{\hspace*{.1mm}\hspace{\fill}}
\newenvironment{claimproof}{}{\claimproofend\par\vspace{2mm}}

\newcommand{\RP}{\mathbb R\mathrm P}

\newcommand{\fig}[1]{\includegraphics[page=#1]{XG-figs}}
\newcommand{\hf}{\hspace*{0mm}\hfill\hspace*{0mm}}
\newcommand{\size}[1]{\left|#1\right|}
\newcommand\Set[2] {\left\{{#1}:\,{#2}\right\}}
\newcommand\Setx[1] {\left\{{#1}\right\}}

\newcommand{\opair}[2]{\left\langle #1,#2 \right\rangle}
\newcommand{\col}[1]{\framebox{\footnotesize $#1$}}

\DeclareMathOperator{\KG}{KG}
\DeclareMathOperator{\SG}{SG}
\DeclareMathOperator{\XG}{XG}

\begin{document}%\suppressfloats
\title{\textbf{Edge-critical subgraphs of Schrijver graphs II: The
    general case}}
\author{Tom\'a\v s Kaiser$^{\:1}$\and Mat\v{e}j Stehl\'{\i}k$^{\:2}$}
\date{\today}

\maketitle
\begin{abstract}
  We give a simple combinatorial description of an $(n-2k+2)$-chromatic
  edge-critical subgraph of the Schrijver graph $\SG(n,k)$, itself
  an induced vertex-critical subgraph of the Kneser graph $\KG(n,k)$.
  This extends the main result of [J.\ Combin.\ Theory Ser.\
  B 144 (2020) 191--196] to all values of $k$, and sharpens the
  classical results of Lov\'asz and Schrijver from the 1970s.
\end{abstract}
\footnotetext[1]{Department of Mathematics and European Centre of
  Excellence NTIS (New Technologies for the Information Society),
  University of West Bohemia, Pilsen, Czech Republic. E-mail:
  \texttt{kaisert@kma.zcu.cz}. Supported by project GA20-09525S of
  the Czech Science Foundation.}%
\footnotetext[2]{Laboratoire G-SCOP, Univ.\ Grenoble Alpes,
  France. E-mail: \texttt{matej.stehlik@grenoble-inp.fr}.  Partially
  supported by ANR project GATO (ANR-16-CE40-0009-01).}%

\section{Introduction}
\label{sec:introduction}

Given integers $k \geq 1$ and $n \geq 2k$,
the \emph{Kneser graph} $\KG(n,k)$ is defined as follows:
the vertices are all the $k$-element subsets of $[n]=\Setx{1, \ldots, n}$,
and the edges are the pairs of disjoint subsets.
A famous conjecture of Kneser~\cite{Kne55}, proved by Lov\'asz~\cite{Lov78}, states that
$\KG(n,k)$ is $(n-2k+2)$-chromatic. Schrijver~\cite{Sch78} sharpened
the result by identifying the elements of $[n]$ with the vertices
of the $n$-cycle $C_n$, and showing that the \emph{Schrijver graph}
$\SG(n,k)$ --- the subgraph of $\KG(n,k)$ induced by the vertices
containing no pair of adjacent elements of $C_n$ ---
is also $(n-2k+2)$-chromatic.
Moreover, Schrijver proved that $\SG(n,k)$ is \emph{vertex-critical}, i.e., the
removal of any vertex decreases the chromatic number.

There is a stronger (and arguably, more natural) notion of
criticality: a graph is said to be \emph{edge-critical}, or simply
\emph{critical}, if the removal of any edge decreases the chromatic
number --- in other words, if any proper subgraph (not necessarily
induced) has a smaller chromatic number than the graph itself.

The Schrijver graph $\SG(n,k)$ is not edge-critical, unless $k=1$ or $n=2k+1$.
This prompts the following natural question: can we give a simple combinatorial
description of an $(n-2k+2)$-chromatic edge-critical subgraph of
$\SG(n,k)$?

In a recent paper~\cite{KS20}, such a construction was given for the
case $k=2$. Here we extend the construction to all values of $k$,
thereby sharpening Schrijver's theorem.

An edge $AB$ of $\SG(n,k)$ is said to be \emph{interlacing} if the
elements of $A$ and $B$ alternate as we go round $C_n$. Simonyi and
Tardos~\cite{ST19+} recently proved that any edge of $\SG(n,k)$ whose
removal decreases the chromatic number is interlacing. Thus,
a tempting candidate for an $(n-2k+2)$-chromatic edge-critical
subgraph of $\SG(n,k)$ might be the spanning subgraph formed by the
interlacing edges. However, Litjens et al.~\cite{LPSV18+} have shown
that this graph has chromatic number $\lceil n/k \rceil$, so
interlacing edges are much too restrictive.

We introduce instead the notion of \emph{almost-interlacing} edges (we
postpone the definition to Section~\ref{sec:definition}), and
define $\XG(n,k)$ to be the spanning subgraph of $\SG(n,k)$ formed by
the almost-interlacing edges. The main result of
this paper is the following theorem:
\begin{theorem}\label{thm:main}
  For every $k \geq 1$ and every $n \geq 2k$, $\chi(\XG(n,k))=n-2k+2$.
  Moreover, $\XG(n,k)$ is edge-critical.
\end{theorem}

We remark that the definition of almost-interlacing edges is particularly
simple for the case $k=2$. Indeed,
almost-interlacing edges of $\SG(n,2)$ correspond to crossing and
transverse edges defined in~\cite{KS20}, so the graph $\XG(n,2)$ is precisely
the graph $G_n$ studied in~\cite{KS20}.

In a forthcoming paper, we will relate the graph $\XG(n,k)$ to the
graphs studied in~\cite{KS17}, and show that $\XG(n,k)$ is a
quadrangulation of $\RP^{n-2k}$ (see~\cite{KS15} for a
definition). In conjunction with the results from~\cite{KS15}, this
gives a new proof of the first part of Theorem~\ref{thm:main}.

For terminology not defined here, we refer the reader to Bondy and Murty~\cite{BM08}.

The paper is structured as follows. Preliminary definitions and
observations are collected in
Section~\ref{sec:preliminaries}. Section~\ref{sec:definition} gives
the definition of the graph $XG(n,k)$. The chromatic number of this
graph is determined in Section~\ref{sec:chi}, and the graph is shown
to be edge-critical in Section~\ref{sec:critical}.

\section{Preliminaries}
\label{sec:preliminaries}

Let $C_n$ be the $n$-cycle with vertex set $[n] = \Setx{1,\dots,n}$
and edges between consecutive integers as well as between $1$ and
$n$. The vertices of the Schrijver graph $\SG(n,k)$ mentioned in
Section~\ref{sec:introduction} are independent sets in $C_n$ of size
$k$; two such sets are adjacent in $\SG(n,k)$ if they are disjoint.

We usually visualise $C_n$ in such a way that the vertices $1,\dots,n$
appear clockwise in the given order. The vertices of $C_n$ will be
referred to as \emph{elements} to distinguish them from the vertices
of $\SG(n,k)$ or of the graph $\XG(n,k)$ we will shortly define. Any
arithmetic operations with the elements are performed modulo the
equality $n + 1 = 1$.

Our arguments frequently use intervals in $C_n$. For $a,b\in [n]$, the
interval $[a,b]$ is the set $\Setx{a,a+1,\dots,b}$. Thus, $[a,b]$
consists of $a$ and the elements following $a$ clockwise up to $b$. In
case $b=a-1$, the interval $[a,b]$ contains all elements of
$[n]$. By a slight abuse of this notation, we will also write $[0,n]$
for the set $\Setx{0,\dots,n}$.

Open or half-open versions of intervals, namely $(a,b)$, $[a,b)$ or
$(a,b]$, are defined as expected: for instance, $[a,b) = [a,b-1]$. All
of the following definitions are modified for these other versions of
intervals in a straightforward way.

If $X\subseteq [n]$, it will be convenient to let
$[a,b]_X = [a,b]\cap X$. The set carries a natural ordering given by
the interval; thus, for instance, the \emph{first} element of
$[a,b]_X$ is the element of this set encountered first when moving
clockwise from $a$ to $b$.

To distinguish ordered pairs from open intervals, we use the notation
$\opair ab$ for an ordered pair consisting of elements $a$ and
$b$. For a set $X\subseteq [n]$, we say that the pair $\opair ab$ is
\emph{$X$-consecutive} if $a,b\in X$ are distinct and $(a,b)_X = \emptyset$.

If $I$ is an interval in $[n]$, we say that disjoint subsets $A,B$ of
$[n]$ \emph{alternate on} $I$ if the elements of $A$ alternate with
those of $B$ as we follow $C_n$ from the start to the end of $I$. Sets
$A,B$ which alternate on $[n]$ are said to form an \emph{interlacing
  pair}.

A crucial notion for our construction is that of an admissible
interval. For disjoint subsets $A,B$ of $[n]$, an interval $[d,c]$ is
\emph{weakly $AB$-admissible} if
\begin{equation*}
  \size{[d,c]_A} = \size{[d,c]_B} = c.
\end{equation*}
Furthermore, a weakly $AB$-admissible interval $[d,c]$ is
\emph{$AB$-admissible} if
\begin{equation*}
  c,d\notin A\cup B.
\end{equation*}
We extend these notions to open or half-open intervals such as $(d,c)$
or $(d,c]$ in precisely the same way, just replacing $[d,c]$ with the
interval in question. 

Let us examine some basic properties of weakly $AB$-admissible
intervals, where $A,B$ are disjoint subsets of $[n]$, each of size
$k$. It is not yet required at this point that $A$ and $B$ be
independent in $C_n$, so we may view $AB$ as an edge of the Kneser
graph $\KG(n,k)$.

\begin{observation}\label{obs:sep}
  If $AB$ is an edge of $\KG(n,k)$ and $[d,c]$ is a weakly
  $AB$-admissible interval, then $c\leq k < d$.
\end{observation}
\begin{proof}
  Note first that $c\leq k$ follows directly from the definition of
  weakly $AB$-admissible interval. Since $A$ and $B$ are disjoint, we
  have $\size{[d,c]_{A\cup B}} = 2c$. It follows that $d > c$, for
  otherwise $2c\leq \size{[d,c]} \leq c$, leading to a contradiction as
  $c \geq 1$. Now
  \begin{equation*}
    2k = \size{A\cup B} = \size{[d,c]_{A\cup B}} + \size{(c,d)_{A\cup B}} \leq 2c + (d-c-1) = c+d-1,
  \end{equation*}
  and since $c \leq k$, we must have $d > k$.
\end{proof}

Another basic property of weakly $AB$-admissible intervals is that
they are nested, as shown by the first part of the
following lemma:
\begin{lemma}
  \label{l:intervals}
  Let $AB$ be an edge of $\KG(n,k)$ and let $[d,c]$, $[d',c']$ be
  weakly $AB$-admissible intervals. Then the following hold:
  \begin{enumerate}[label=(\roman*)]
  \item $[d',c']\subseteq [d,c]$ or vice versa,
  \item if $[d',c']\subseteq [d,c]$ and $c' < c$, then the set
    $[d,d')_{A\cup B}$ is nonempty; if, moreover, $[d,c]$ is
    $AB$-admissible, then $\size{(d,d')_{A\cup B}} \geq 2$.
  \end{enumerate}
\end{lemma}
\begin{proof}
  (i) Suppose that the claim does not hold. By
  Observation~\ref{obs:sep} and by symmetry, we may assume that
  $c < c' < d < d'$. Then
  \begin{align*}
    2c' = \size{[d',c']_{A\cup B}} = \size{[d',c]_{A\cup B}} +
    \size{(c,c']_{A\cup B}} \leq 
    2c + (c'-c)
  \end{align*}
  implying $c' \leq c$, a contradiction.

  (ii) Our assumptions imply $c' < c < d \leq d'$. We have
  \begin{align}
    2c = \size{[d,c]_{A\cup B}}
    &= \size{[d,d')_{A\cup B}} +
      \size{[d',c']_{A\cup B}} + \size{(c',c]_{A\cup B}} \nonumber\\
    &\leq \size{[d,d')_{A\cup B}} + 2c' + (c-c'), \label{eq:cross}
  \end{align}
  so $\size{[d,d')_{A\cup B}} \geq c-c' \geq 1$.

  If $c,d\notin A\cup B$, then the $(c-c')$ term in~\eqref{eq:cross}
  improves to $(c-c'-1)$ and furthermore, we can write
  $(d,d')_{A\cup B}$ in place of $[d,d')_{A\cup B}$. The second
  assertion follows.
\end{proof}

We conclude this section by the definition of switching, used in
Section~\ref{sec:definition} to introduce the graph
$\XG(n,k)$. Suppose that $c,d\in [n]$. \emph{Switching} at $[d,c]$ is
the operation transforming any pair $AB$ of subsets of $[n]$ to
another such pair $A'B'$ defined as follows:
\begin{align*}
  A' &= A \mathbin\triangle [d,c]_{A\cup B},\\
  B' &= B \mathbin\triangle [d,c]_{A\cup B},
\end{align*}
where $\mathbin\triangle$ denotes symmetric difference. The pair $A'B'$ is the
\emph{result} of the switching.

It is easy to see that if $AB$ is an edge of $\KG(n,k)$, then
the result of switching $AB$ at a weakly $AB$-admissible interval is
again an edge of $\KG(n,k)$. A similar statement holds for $\SG(n,k)$
and switching at an $AB$-admissible interval.

\emph{Switching along a sequence} $([d_i,c_i])_{i\in [m]}$ of
intervals means switching at $[d_1,c_1],\dots,[d_m,c_m]$ in this
order. (Switching along an empty sequence is the identity operation on
pairs.)

Under an admissibility assumption, switching along a sequence of
intervals maps any edge of the
Schrijver graph $\SG(n,k)$ to an edge:
\begin{observation}\label{obs:admissible}
  Let $AB$ be an edge of $\SG(n,k)$ and let $A'B'$ be the pair obtained
  by switching $AB$ along a sequence $S$ of $AB$-admissible
  intervals. The following holds:
  \begin{enumerate}[label=(\roman*)]
  \item $A'B'$ is again an edge of $\SG(n,k)$,
  \item any weakly $AB$-admissible interval is weakly
    $A'B'$-admissible and vice versa.
  \end{enumerate}
\end{observation}

\section{Definition of $\XG(n,k)$}
\label{sec:definition}

In this section, we define the graph $\XG(n,k)$. Let $k\geq 1$ and $n\geq
2k$. The vertex set of $\XG(n,k)$ coincides with that of $\SG(n,k)$, so
the vertices of $\XG(n,k)$ are all $k$-element independent sets of
$C_n$. The edges of $\XG(n,k)$ are all the almost-interlacing pairs,
defined as follows.

A pair $AB$ of vertices, where $A\cap B = \emptyset$, is
\emph{almost-interlacing} if there exists a set
$X = C\cup D \subseteq [n]$ such that $C = \Setx{c_1,\dots,c_m}$ and
$D = \Setx{d_1,\dots,d_m}$, with the following properties:
\begin{enumerate}[label=(\arabic*)]
\item $1\leq c_1 < c_2 < \dots < c_m \leq k-1$,
\item $k+1 \leq d_m < d_{m-1} < \dots < d_1 \leq n$,
\item each interval $[d_i,c_i]$ is $AB$-admissible,
\item switching along the sequence $([d_i,c_i])_{i\in [m]}$ changes
  $AB$ to an interlacing pair.
\end{enumerate}
Any set $X$ satisfying this definition is called an
\emph{$AB$-alternator}. We often write it as $C\cup D$, with $C$ and
$D$ as in the definition. The elements in $C$ are the \emph{control
  elements} of the $AB$-alternator, the elements $c_i$ and $d_i$
($i\in [m]$) \emph{correspond} to each other, and pairs
$\opair {c_i} {d_i}$ ($i\in [m]$) are the \emph{control pairs} of the
$AB$-alternator.

Observe that $\XG(n,k)$ is a spanning subgraph of $\SG(n,k)$. Any pair
of vertices $AB$ that is an interlacing pair is an edge of $\XG(n,k)$,
since in this case the empty set is trivially an $AB$-alternator.

Another example is shown in Figure~\ref{fig:example1}, depicting an
edge $AB$ of $\SG(16,4)$. The set $\Setx{2,3,7,11}$ is an
$AB$-alternator, so $A$ and $B$ are adjacent in $\XG(16,4)$. There is
only one other $AB$-alternator, namely $\Setx{2,3,7,10}$.

\begin{figure}
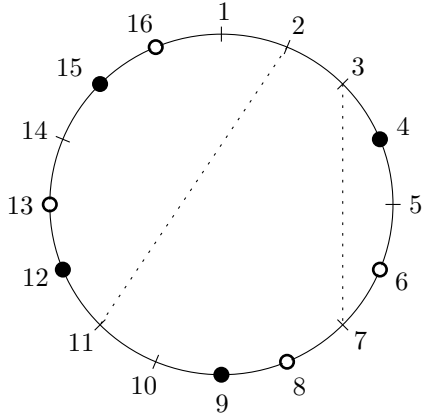

  \centering\fig2
  \caption{Vertices $A = \Setx{4,9,12,15}$ (black dots) and
    $\Setx{6,8,13,16}$ (white dots) of $\XG(16,4)$ forming an
    almost-interlacing pair. The elements not in $A\cup B$ are shown
    as tick marks. Dotted lines mark the control pairs of the
    $AB$-alternator $\Setx{2,3,7,11}$. Similar conventions are used in
    the other figures.}
  \label{fig:example1}
\end{figure}

Let us consider the special case of the definition for $k=2$. (See
Figure~\ref{fig:k2} for an illustration.) Let $AB$ be an edge of
$\SG(n,2)$. We may assume that $A=\Setx{a_1,a_2}$, $B=\Setx{b_1,b_2}$,
where $a_1 < a_2$, $b_1 < b_2$ and $a_1 < b_1$.  Possible
$AB$-alternators are $\emptyset$ (in which case $AB$ is an interlacing
pair), or a set $\Setx{1,d}$, disjoint from $A\cup B$, such that
$[d,1]_{A\cup B} = \Setx{a_2,b_2}$ (which is easily seen to be
equivalent to $1 < a_1 < b_1 < b_2 < a_2$). In the paper~\cite{KS20},
pairs of these two types are referred to as crossing and transverse
pairs, respectively, and they coincide with the edges of the graph
studied in that paper (denoted by $G_n$). Thus, as noted in
Section~\ref{sec:introduction}, the present definition specialises to
the one of~\cite{KS20} for $k=2$.

\begin{figure}
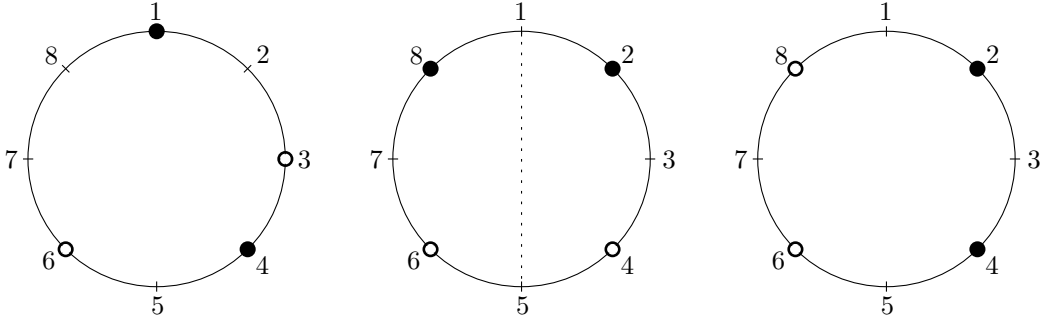

  \fig3\hf\fig4\hf\fig5
  \caption{Examples of edges in $\XG(8,2)$ (left and center) and a
    non-edge in $\XG(8,2)$ (right). The dotted line in the center
    picture shows the only control pair of the $AB$-alternator
    $\Setx{1,5}$.}
  \label{fig:k2}
\end{figure}

Let us add some comments on the definition of edges of $\XG(n,k)$. Note
that in condition (1), the bound $c_m \leq k$ is trivial (and stated
in Observation~\ref{obs:sep}), so (1) just strengthens this bound by
one. Furthermore, the bound $k+1\leq d_m$ in condition (2) is actually
superfluous (though we include it for clarity) as it also follows from
Observation~\ref{obs:sep}. Using Lemma~\ref{l:placement}(i) below, the
bounds in condition (2) can be strengthened to $d_m \geq k+2$ and
$d_1 \leq n-2$.

We will now describe an algorithm that finds an $AB$-alternator
$C\cup D$ if it exists, where $AB$ is an edge of $\SG(n,k)$. It may be
helpful to consult Figure~\ref{fig:example1} for an
illustration. First we need another lemma.

\begin{lemma}\label{l:placement}
  Let $C\cup D$ be an $AB$-alternator for an edge $AB$ of $\SG(n,k)$
  such that $AB$ is not interlacing. The following hold:
  \begin{enumerate}[label=(\roman*)]
  \item if $\opair xy$ is a $(D\cup\Setx{k,n})$-consecutive pair other
    than $\opair nk$, then the size of $[x,y]_{A\cup B}$ is at least
    $2$,
  \item if $\opair xy$ is an $(A\cup B)$-consecutive pair, then
    $\size{(x,y)_{C\cup D}}$ is odd if and only if $x,y\in A$ or
    $x,y\in B$.
  \end{enumerate}
\end{lemma}
\begin{proof}
  Let $D = \Setx{d_1,\dots,d_m}$ with $d_1 > \dots > d_m$. Since $AB$
  is not interlacing, we have $m\geq 1$. For $i\in [m]$, let $c_i$ be
  the control element corresponding to $d_i$.

  (i) If $x,y\in D$, then the assertion follows from
  Lemma~\ref{l:intervals}(ii) and the fact that each of the intervals
  $[d_i,c_i]$ is $AB$-admissible.

  For the pair $\opair k{d_m}$, we can write
  \begin{align*}
    2k = \size{A\cup B}
    &= \size{[d_m,c_m]_{A\cup B}} +
      \size{(c_m,k)_{A\cup B}} + \size{[k,d_m)_{A\cup B}}\\
    &\leq 2c_m + (k-c_m-1) + \size{[k,d_m]_{A\cup B}},
  \end{align*}
  so $\size{[k,d_m]_{A\cup B}} \geq k-c_m+1 \geq 2$ since $c_m\leq k-1$.

  Similarly, for the pair $\opair{d_1}n$, we have
  \begin{equation*}
    2c_1 = \size{[d_1,n]_{A\cup B}} + \size{[1,c_1]_{A\cup B}} \leq
    \size{[d_1,n]_{A\cup B}} + (c_1 - 1)
  \end{equation*}
  (using the fact that $c_1\notin A\cup B$), and we find that
  $\size{[d_1,n]_{A\cup B}} \geq c_1 + 1 \geq 2$.

  (ii) Let us say that a subset of $[n]$ is \emph{separating} if it
  contains exactly one of $x$ and $y$. Let $s$ be the number of
  intervals $[d_i,c_i]$ ($i\in [m]$) which are separating. Observe
  that $s$ has the same parity as $\size{(x,y)_{C\cup D}}$.

  For $0\leq j\leq m$, let $A_jB_j$ be the pair obtained from $AB$ by
  switching along $([d_i,c_i])_{i\in [j]}$; in particular,
  $A_0B_0 = AB$. For $j > 0$, it is not hard to see that $A_j$ is
  separating if and only if exactly one of $A_{j-1}$ and $[d_j,c_j]$
  is separating. Now since $A_mB_m$ is an interlacing pair, $A_m$ is
  not separating. It follows that either $A$ is separating and $s$ is
  even, or $A$ is not separating and $s$ is odd. Since $A$ is not
  separating if and only if $x,y\in A$ or $x,y\in B$, and by the above
  observation on the parity of $s$, this implies part (ii).
\end{proof}

Let us return to the task of finding an $AB$-alternator for a given
edge $AB$ of $\SG(n,k)$. Consider any $(A\cup B)$-consecutive pair
$\opair xy$ with $x,y\in [k,n]$. If $x,y\in A$ or $x,y \in B$, then by
Lemma~\ref{l:placement}(ii), our set $D$ needs to contain an element
in $(x,y)$. The latter interval is nonempty since each of $A$ and $B$
is independent in $C_n$. Furthermore, by Lemma~\ref{l:placement}(i),
$D$ must contain exactly one element from this interval. The choice of
the element from $(x,y)$ is arbitrary; in fact, we will see that this
is the only choice we have in the process. In the example of
Figure~\ref{fig:example1}, the set $D$ must include the element $7$
and one element from $\Setx{10,11}$.

Similarly to the above, Lemma~\ref{l:placement}(ii) and (i) implies
that if exactly one of $x,y$ is in $A$, then $(x,y)_D$ must be empty,
because its size is even and at most one. Finally, by
Lemma~\ref{l:placement}(i), $D$ contains no element between $k$ and
the first element of $[k,n]_{A\cup B}$, nor between the last element
of the latter set and $n$.

Summing up, $D$ is obtained by choosing exactly one element in each
interval $(x,y)$ with $\opair xy$ an $(A\cup B)$-consecutive pair with
$x,y\in [k,n]$ and either $x,y\in A$ or $x,y\in B$. Let
$D = \Setx{d_1,\dots,d_m}$ for some such choice. (Thus, for the pair
in Figure~\ref{fig:example1}, $D$ equals $\Setx{7,10}$ or
$\Setx{7,11}$.)

We will show that this determines the set $C$ whenever there exists an
$AB$-alternator. The following lemma provides a tool.

\begin{lemma}\label{l:depth}
  Let $d\in [k,n]$ and let $X$ be a vertex of $\SG(n,k)$. There is at
  most one element $c\in [k-1]$ such that $\size{[d,c]_X} = c$ and
  $c\notin X$.
\end{lemma}
\begin{proof}
  For $x\in [k-1]$, let
  \begin{equation*}
    f(x) = \size{[d,x]_X}-x.
  \end{equation*}
  The function $f$ is non-increasing. For each $x\in [k-2]$, we have
  \begin{equation*}
    f(x+1) =
    \begin{cases}
      f(x) & \text{if $x+1\in X$,}\\
      f(x)-1 & \text{otherwise.}
    \end{cases}
  \end{equation*}
  Thus, if $f(x) = f(x+1)$ and $x\leq k-3$, then $f(x+1) > f(x+2)$ by
  the independence of $A$. It follows that we have $f(x) = 0$ for at
  most two values of $x$. Supposing (for the sake of a contradiction)
  that the lemma does not hold, there are two such values, say $c$ and
  $c+1$, where $c\in [k-2]$. Since $f(c+1) = f(c)$, we have
  $c+1\in X$, so $c+1$ does not satisfy the conditions, a
  contradiction.
\end{proof}

For each $i\in [m]$, $C$ has to contain an element $c_i$ such that
$[d_i,c_i]$ is $AB$-admissible. Since $c_i$ has to satisfy the
condition of Lemma~\ref{l:depth} with $X=A$, there is at most one such
element. Furthermore, $c_i$ is independent of the choice of $d_i$:
more precisely, if $\opair xy$ is the $(A\cup B)$-consecutive pair
such that $d_i\in (x,y)$, and if $d'_i\in (x,y)$, then
$[d'_i,c]_{A\cup B} = [d_i,c]_{A\cup B}$ for any $c\in [k-1]$. It
follows that if an $AB$-alternator does exist, then each element of
$C$ is uniquely determined by Lemma~\ref{l:depth}. Our algorithm
returns $C\cup D$ when this is the case, and reports that there is no
$AB$-alternator otherwise. (In the example of
Figure~\ref{fig:example1}, we have $c_1 = 2$ and $c_2 = 3$, so one of
the sets $\Setx{2,3,7,10}$ or $\Setx{2,3,7,11}$ is returned.)

To obtain a unique choice for the $AB$-alternator when it exists, we
impose the extra condition that for each $i\in [m]$,
$d_i + 1\in A\cup B$. This amounts to choosing the largest possible
element for each $d_i$. The resulting $AB$-alternator is called
\emph{standard}. Speaking of the control elements or control pairs for
the edge $AB$, we mean the control elements or pairs of the standard
$AB$-alternator.

\section{Chromatic number}
\label{sec:chi}

In this section, we prove the first part of Theorem~\ref{thm:main} ---
namely, that $\chi(\XG(n,k))=n-2k+2$ for every $k \geq 1$ and every $n \geq 2k$.
It is enough to prove the inequality $\chi(\XG(n,k)) \geq n-2k+2$,
the other inequality being a direct consequence of the fact that $\XG(n,k)$
is a subgraph of $\KG(n,k)$.

The case $k=2$ of Theorem~\ref{thm:main} was proved in~\cite{KS20} using the
so-called Mycielski construction. Here we prove the general case using
the same idea, but rely instead on the \emph{generalised} Mycielski
construction, introduced by Stiebitz~\cite{Sti85} (see
also~\cite{GJS04,SS89}).

Given a graph $G=(V,E)$ and an integer $r \geq 1$, the graph $M_r(G)$
has vertex set $(V \times [0,r-1]) \cup \Setx{z}$, and there is an
edge $(u,0)(v,0)$ and $(u,i)(v,i+1)$ (for every $i \in [0,r-2]$)
whenever $uv \in E$, and an edge $(u,r-1)z$ for all $u \in V$.  The
construction is illustrated in Figure~\ref{fig:gen-mycielski}.

\begin{figure}
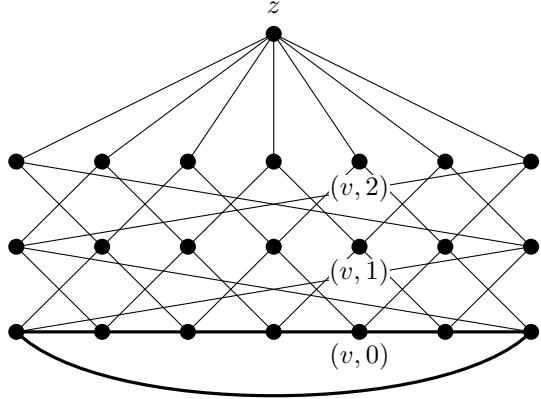

  \centering\fig6
  \caption{The generalised Mycielski construction applied to $C_7$
    (bold) resulting in the graph $M_3(C_7)$.}
  \label{fig:gen-mycielski}
\end{figure}

For every integer $t \geq 2$, we denote by $\mathcal M_t$ the
set of all `generalised Mycielski graphs' obtained from $K_2$ by
$t-2$ iterations of $M_r(\cdot)$, where the value of $r$ can
vary from iteration to iteration. That is, $H \in \mathcal M_t$
if and only if there exist integers $r_1, r_2, \ldots, r_{t-2} \geq 1$
such that
\begin{equation*}
  H \cong M_{r_{t-2}}(M_{r_{t-3}}(\ldots M_{r_2}(M_{r_1}(K_2)) \ldots )).
\end{equation*}

Using topological methods, Stiebitz~\cite{Sti85}
(see also~\cite{GJS04,Mat03}) proved the following result.
A `discrete' proof, based on a combinatorial lemma of Fan, can
be found in~\cite{MS19}.

\begin{theorem}[Stiebitz~\cite{Sti85}]
\label{thm:stiebitz}
  If $G \in \mathcal M_t$, then $\chi(G) = t$.
\end{theorem}

We now come to the key lemma of this section.

\begin{lemma}
  \label{lem:homomorphism}
  For every $k \geq 1$ and every $n\geq 2k$, $M_k(\XG(n-1,k))$ is homomorphic
  to $\XG(n,k)$.
\end{lemma}

\begin{proof}
  We shall explicitly describe a homomorphism $f$ from
  $M_k(\XG(n-1,k))$ to $\XG(n,k)$. Let $A$ be a vertex of $\XG(n-1,k)$
  and let $(A,0), \dots,(A,k-1)$ be its copies in
  $M_k(\XG(n-1,k))$. In order to keep all vertex names capitalised, we
  choose to denote the vertex $z$ in the generalised Mycielski
  construction by $Z$.
  
  Suppose that $A = \Setx{a_1,\dots,a_k}$, where $a_1 < \dots <
  a_k$. 
  Let $0 \leq i \leq k$. We define the set $\Lambda_{n,i}\subseteq [n]$
  as follows:
  \begin{equation*}
    \Lambda_{n,i} =
    \begin{cases}
      \Setx{n-i+1,n-i+3,\dots,n}\cup\Setx{2,4,\dots,i-1} & \text{if $i$ is odd,}\\
      \Setx{n-i+1,n-i+3,\dots,n-1}\cup\Setx{1,3,\dots,i-1} & \text{if $i$
        is even.}
    \end{cases}
  \end{equation*}
  Thus, for instance, $\Lambda_{n,0}=\emptyset$, $\Lambda_{n,1}=\Setx{n}$
  and $\Lambda_{n,2}=\Setx{1,n-1}$.
  
  We will now define a map $f:\,V(M_k(\XG(n-1,k))) \to V(\XG(n,k))$.
  Given a vertex $A$ of $\XG(n-1,k)$ and an integer $j \in [k]$, let
  $A^j=[d,j]_A$, where $d$ is the maximum integer such that
  $|[d,j]_A|=j$. Furthermore, let $A^0 = \emptyset$. We set
  \begin{align*}
    f:\,(A,j) &\mapsto (A \setminus A^j) \cup \Lambda_{n,j},
                \text{ where $0\leq j\leq k-1$,} \\
      Z     &\mapsto \Lambda_{n,k}.
  \end{align*}
  
  Note that the image of $f$ is contained in the vertex set of
  $\XG(n,k)$. Informally, $f(A,j)$ can be seen as the result of the
  following process: viewing $A$ as a subset of $V(C_n)$, $A^j$
  consists of the $j$ elements of $A$ that are closest to $j$
  counterclockwise; push them clockwise in such a way that the first
  one stops at $j$ and the remaining ones are tightly packed (still
  forming an independent set), and rotate them back by one
  element. The other $k-j$ elements of $A$ are not affected.

  To verify that $f$ is a homomorphism, it is enough to check that $f$
  maps edges of $M_k(\XG(n-1,k))$ to edges of $\XG(n,k)$. Fix an
  arbitrary edge $AB$ of $\XG(n-1,k)$, and let
  $C\cup D \subseteq [n-1]$ be the standard $AB$-alternator. Let
  $\Set{\opair {c_i}{d_i}}{i\in [m]}$ be its set of control pairs.

  We will show that $f$ maps the edges $(A,0)(B,0)$,
  $(A,j)(B,j+1)$ (for any $j \in [0,k-1]$), as well as
  $(A,k-1)Z$, to edges of $\XG(n,k)$, by finding an appropriate
  alternator $C'\cup D'$.
  
  First, consider the edge $(A,0)(B,0)$ of $M_k(\XG(n-1,k))$.
  Since $f((A,0))=A$ and $f((B,0))=B$, the required alternator is
  obtained by taking $C'=C$ and $D'=D$. (Note that the definition is
  still satisfied if $A$ and $B$ are viewed as vertices of $\XG(n,k)$
  rather than $\XG(n-1,k)$.)

  Edges of type $(A,k-1)Z$ are another easy case: we have
  $f(Z) = \Lambda_{n,k}$ and $f((A,k-1))$ contains $\Lambda_{n,k-1}$
  as a subset, which means that $f((A,k-1))f(Z)$ must actually be an
  interlacing pair, and hence an edge of $\XG(n,k)$ (with empty
  alternator).

  It remains to consider the edge $(A,j)(B,j+1)$, where
  $j\in [0,k-1]$. Let $A' = f((A,j))$ and $B' = f((B,j+1))$. The sets
  $A'$ and $B'$ are disjoint since $A\cap B = \emptyset$ and
  $\Lambda_{n,j}\cap\Lambda_{n,j+1} = \emptyset$.

  Given $r\in [0,m]$, let $A_rB_r$ be the pair obtained from $AB$ by
  switching along $([d_i,c_i])_{i\in [r]}$. Since $A_mB_m$ is
  interlacing, there is $d\in [k+1,n]$ such that $[d,j+1]$ is weakly
  $A_mB_m$-admissible. Choose $d$ as maximal with this property. By
  Observation~\ref{obs:admissible}(ii), $[d,j+1]$ is weakly
  $AB$-admissible.

  For any $i\in [m]$, we have $[d_i,c_i] \subseteq [d,j+1]$ or vice
  versa by Lemma~\ref{l:intervals}(i). If there is $t\in [m]$ such
  that $c_t < j+1$, then let $t$ be maximal with this property;
  otherwise, let $t = 0$.

  We now aim to show that the pair $A'B'$ is, in a sense, not too
  different from $A_tB_t$.

  Let $X,Y$ be disjoint vertices of the graph $H$ (that is, vertices
  such that $X\cap Y=\emptyset$), where $H$ is either $\XG(n-1,k)$ or
  $\XG(n,k)$. Let $I$ be the interval $[d,j+1] \subseteq V(C_n)$,
  where $d$ is as above. (Thus, $n\in I$ even if $H$ is $\XG(n-1,k)$.)

  Let us say that the pair $XY$ is
  \emph{nice} if the following hold:
  \begin{enumerate}[label=(N\arabic*)]
  \item $X\setminus I = A \setminus I$ and
    $Y\setminus I = B \setminus I$,
  \item the sets $X$ and $Y$ alternate on $I$ and the first element of
    $I\cap (X\cup Y)$ belongs to $X$ if and only if the first element
    of $I\cap (A\cup B)$ belongs to $A$.
  \end{enumerate}

  \begin{claim}\label{cl:nice1}
    The pair $A_tB_t$ is nice.
  \end{claim}
  \begin{claimproof}
    Condition (N1) in the definition follows from the fact that for
    each of the intervals $[d_i,c_i]$ with $i \leq t$, we have
    $c_i < j+1$ and therefore $[d_i,c_i] \subseteq I$ by
    Lemma~\ref{l:intervals}(i). Thus, switching at such intervals does
    not affect the elements outside $I$.

    Let us verify condition (N2). Since $A_mB_m$ is an interlacing
    pair and $I\subseteq [d_i,c_i]$ for any $i > t$, $A_t$ and $B_t$
    must alternate on $I$. For the rest of condition (ii), we may
    assume that $t > 0$. Let $x$ be the first element of
    $I\cap (A\cup B)$; since $A\cup B = A_t\cup B_t$, this is also the
    first element of $I\cap (A_t\cup B_t)$. By
    Lemma~\ref{l:intervals}(ii), $x$ is not contained in $[d_t,c_t]$
    (nor in any $[d_i,c_i]$ with $i < t$), and therefore $x\in A_t$ if
    and only if $x\in A$. This concludes the proof of the claim.
  \end{claimproof}

  \begin{claim}\label{cl:nice-edge}
    Any nice pair $XY$ of disjoint vertices of $\XG(n,k)$ forms an edge
    of $\XG(n,k)$.
  \end{claim}
  \begin{claimproof}
    It is clear from the definition of nice pair that $XY$ can be
    obtained from (the nice pair) $A_tB_t$ by first extending the
    underlying cycle $C_{n-1}$ to $C_n$ (just inserting the element
    $n$) and then moving the elements of $X\cup Y$ within $I$ without
    changing their order on $C_n$.

    It follows that switching along
    $([d_{t+1},c_{t+1}],\dots,[d_m,c_m])$ changes $XY$ to an
    interlacing pair, just as in the case of $A_tB_t$. (Recall that
    $I$ is a subset of each of these intervals by the choice of $t$.)
    Summing up,
    $\Setx{c_{t+1},\dots,c_m}\cup\Setx{d_{t+1},\dots,d_m} \subseteq
    [n]$ is an $XY$-alternator.
  \end{claimproof}
  
  The following claim relates the above observations to $A'B'$.

  \begin{claim}\label{cl:nice2}
    One of the following conditions holds:
    \begin{enumerate}[label=(\roman*)]
    \item $A'B'$ is a nice pair,
    \item the interval $I$ is $A'B'$-admissible and the pair $A''B''$
      obtained by switching $A'B'$ at $I$ is nice.
    \end{enumerate}
  \end{claim}
  \begin{claimproof}
    First of all, observe that since $I$ is weakly $AB$-admissible,
    both $A^j$ and $B^{j+1}$ are contained in $I$. Furthermore, both
    $\Lambda_{n,j+1}$ and $\Lambda_{n,j}$ are contained in $I$:
    indeed, the weakly $A'B'$-admissible interval $I=[d,j+1]$ must
    satisfy $d \leq n-j$, while at the same time
    $\Lambda_{n,j}\cup \Lambda_{n,j+1} = [n-j,j]$. This proves
    condition (N1) for both of the pairs involved in (i) and (ii).

    We have in fact $B^{j+1} = I\cap B$ and
    $A^j = (I\cap A)\setminus\Setx a$ for some $a\in I\cap A$. There
    are essentially three possibilities for $a$, illustrated in
    Figure~\ref{fig:homo}: if $j+1\notin A$, then $a$ is the first
    element of $[d,j+1]_A$ and it may or may not equal $d$, while if
    $j+1\in A$, then $a = j+1$.

    All the elements of $B^{j+1}$ are replaced in $B'$ by
    $\Lambda_{n,j+1}$; similarly, all the elements of $A^j$ are
    replaced in $A'$ by $\Lambda_{n,j}$. Hence, $A'$ and $B'$
    alternate on $[n-j,j]$, and therefore they alternate on $I$
    regardless of the position of the remaining element $a$ of
    $[d,j+1]_{A'\cup B'}$.

    \begin{figure}
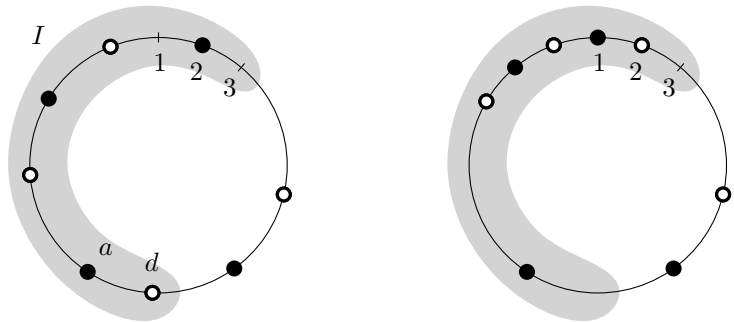
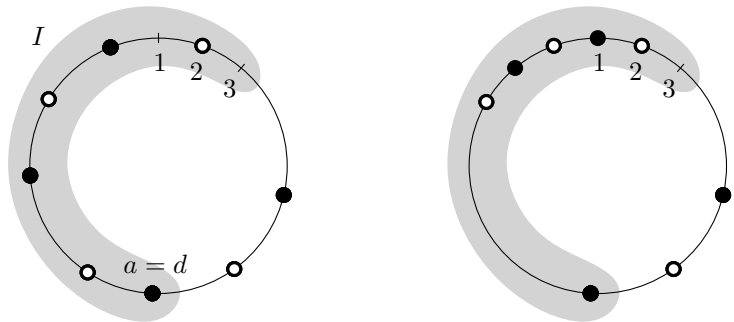
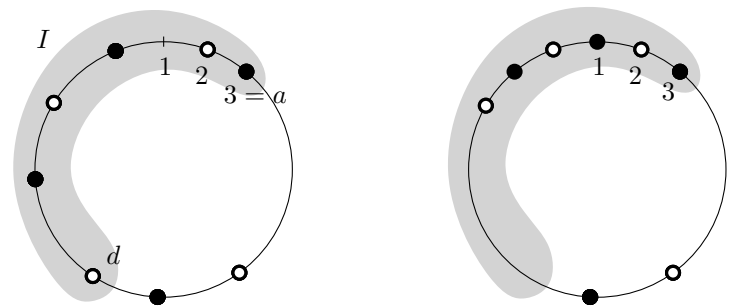

      \centering
      \subfloat[$a$ differs from both $d$ and $j+1$.]{\fig7\hspace{2cm}\fig8}\\[5mm]
      \subfloat[$a=d$.]{\fig9\hspace{2cm}\fig{10}}\\[5mm]
      \subfloat[$a=j+1$.]{\fig{11}\hspace{2cm}\fig{12}}
      \caption{Possible cases in the proof of Claim~\ref{cl:nice2},
        shown for $k=4$, $j=2$ and $AB$ interlacing. The figures on
        the left show the pair $AB$, those on the right show
        $A'B'$. Black dots represent $A$ or $A'$, white dots represent
        $B$ or $B'$, the interval $I=[d,j+1]$ is shown gray. The cases
        are distinguished by the position of the element $a$ of
        $(I\cap A)\setminus A^j$. The pair $A'B'$ is interlacing except in
        (a), in which case a switch at $[d,j+1]$ is needed to make it
        interlacing.}
      \label{fig:homo}
    \end{figure}

    If condition (N2) holds for $A'B'$, then we are done. Assume thus
    that this is not the case. We have $a\neq d$, for otherwise $a$
    would be the first element of both $I\cap (A\cup B)$ and
    $I\cap (A'\cup B')$ while $a\in A\cap A'$, implying (N2). For a
    similar reason (using the fact that $A'B'$ alternates in $I$), we
    find $a\neq j+1$. Consequently, neither $d$ nor $j+1$ belong to
    $A'$. They do not belong to $B'$ either: this is clear in the case
    of $j+1$, and $d\in B'$ would only be possible if $d = n-j$, but
    then $\size{[d,j+1]_A} = j+1$ would force $j+1\in A$ and hence
    $a = j+1$, a contradiction. We have proved that $[d,j+1]$ is
    $A'B'$-admissible.

    Let $x$ be the first element of
    $[d,j+1]_{A'\cup B'} = [d,j+1]_{A''\cup B''}$ and note that $x$
    belongs to $A'$ if and only if it belongs to $B''$. Thus,
    condition (N2) is satisfied for exactly one of the pairs $A'B'$
    and $A''B''$. This proves the claim.
  \end{claimproof}

  Let us finish the proof of the lemma. If condition (i) of
  Claim~\ref{cl:nice2} holds, then $A'B'$ is an edge of $\XG(n,k)$ by
  Claim~\ref{cl:nice-edge}. If condition (ii) holds, then we obtain an
  $A'B'$-alternator by setting $C' = C \cup \Setx{j+1}$ and
  $D' = D \cup \Setx{d}$, completing the discussion for edges of type
  $(A,j)(B,j+1)$ as well as the whole proof.
\end{proof}

We are now ready to prove that $\chi(\XG(n,k)) \geq n-2k+2$.
First, observe that if $G$, $H$ are graphs such that $G$ is
homomorphic to $H$, then $M_k(G)$ is homomorphic to $M_k(H)$.
Hence, by repeated applications of Lemma~\ref{lem:homomorphism},
the graph
\begin{equation*}
  H=M_k(M_k(\ldots M_k(\XG(2k,k)) \ldots )),
\end{equation*}
where $M_k(\cdot)$ is applied $n-2k$ times, is homomorphic to
$\XG(n,k)$. Since $\XG(2k,k)$ is isomorphic to $K_2$,
$H \in \mathcal M_{n-2k}$, so using Theorem~\ref{thm:stiebitz},
we conclude that $\chi(\XG(n,k)) \geq n-2k+2$.

%Since the Schrijver graph $\SG(n,k)$ is vertex-critical, the
%homomorphism of Lemma~\ref{lem:homomorphism} is necessarily
%surjective. In fact, the result that $\XG(n,k)$ is edge-critical,
%proved in the following section, implies that even the mapping between
%the edge sets of the two graphs induced by the homomorphism is
%surjective.

\section{Criticality}
\label{sec:critical}

In this section, we prove the second part of Theorem~\ref{thm:main},
namely that $\XG(n,k)$ is edge-critical. Let $AB$ be an edge of
$\XG(n,k)$ and let $G = \XG(n,k)-AB$. We show that $G$ is
$(n-2k+1)$-colourable.

Let $C\cup D$ be the standard $AB$-alternator, where
$C = \Setx{c_1,\dots,c_m}$, $D = \Setx{d_1,\dots,d_m}$ and
\begin{equation*}
  c_1 < c_2 < \dots < c_m \leq k - 1 < d_m < d_{m-1} < \dots < d_1.
\end{equation*}
The sets $A$, $B$, $C$, $D$ are pairwise disjoint and for $j \in [m]$,
$\size{[d_i,c_i]_A} = \size{[d_i,c_i]_B} = c_i$.

Let $W = A\cup B\cup C\cup D$. We call a vertex of $G$
\emph{essential} if it is contained in $W$ and \emph{inessential}
otherwise. In our analysis, it will be sufficient to concentrate on
essential vertices since each inessential one will get a colour
special to one of its elements outside $W$, and it will be easy to see
that these colour classes are independent sets in $G$.

\begin{lemma}\label{l:disbalance}
  Suppose that $X,Y$ are disjoint vertices of $G$, $c\in [k-1]$ and
  $d\in [n]$ such that $c\neq d$. The pair $XY$ is \emph{not} an edge
  of $G$ if one of the following conditions holds:
  \begin{enumerate}[label=(\roman*)]
  \item $\size{[d,c]_X} > c$ and $\size{[d,c]_Y} < c$, or
  \item $\size{[d,c)_X} > c$ and $\size{[d,c)_Y} < c$.
  \end{enumerate}
\end{lemma}
\begin{proof}
  Assume condition (i). For the sake of a contradiction, assume that
  $XY$ is an edge of $G$, and consider the standard $XY$-alternator
  $C'\cup D'$, where $C' = \Setx{c'_1,\dots,c'_\ell}$ and
  $D' = \Setx{d'_1,\dots,d'_\ell}$.

  Suppose first that
  \begin{equation}\label{eq:comparable}
    \text{for each } j\in [\ell], [d'_j,c'_j] \subseteq [d,c]\text{ or
      vice versa.}
  \end{equation}

  For $0 \leq j \leq\ell$, let $X_jY_j$ be the result of switching
  $XY$ along $([d'_i,c'_i]_{i\in [j]})$. Let
  \begin{equation*}
    \beta(X_jY_j) = \left| \size{[d,c]_{X_j}} - \size{[d,c]_{Y_j}} \right|. 
  \end{equation*}
  Since $X_\ell Y_\ell$ is an interlacing pair, we have
  $\beta(X_\ell Y_\ell) \leq 1$. We claim that for $j > 0$, it holds
  that $\beta(X_jY_j) = \beta(X_{j-1}Y_{j-1})$. This is clear if
  $[d,c] \subseteq [d'_j,c'_j]$, for then the effect of the switch at
  $[d'_j,c'_j]$ within $[d,c]$ is just to interchange membership in
  $X_{j-1}$ and $Y_{j-1}$. On the other hand, if
  $[d'_j,c'_j] \subseteq [d,c]$, then $[d'_j,c'_j]$ is
  $X_{j-1}Y_{j-1}$-admissible by Observation~\ref{obs:admissible}(ii),
  and therefore $[d,c]_{X_j} = [d,c]_{X_{j-1}}$ and similarly
  $[d,c]_{Y_j} = [d,c]_{Y_{j-1}}$. The claim follows.

  Since $XY = X_0Y_0$, we have shown that $\beta(XY)\leq 1$. This
  contradiction with condition (i) implies that our
  assumption~\eqref{eq:comparable} does not hold.

  Thus, let $j$ be the least index such that
  $\size{[d,c] \cap \Setx{c'_j,d'_j}} = 1$.

  Suppose that $d'_j\notin [d,c]$. Since $\size{[d,c]_X} > c$ and
  $c'_j\in [d,c]$, we have $\size{[d'_j,c]_X} > c$. On the other hand,
  $\size{[d'_j,c'_j]_X} = c'_j$, and thus
  \begin{equation}
    \label{eq:1}
    \size{(c'_j,c]_X} \geq c + 1 - c'_j.
  \end{equation}
  Since $X$ is an independent set in $C_n$, we have
  $\size{(c'_j,c]_X} \leq (c-c'_j+1)/2$. Combining this
  with~\eqref{eq:1}, we derive $c < c'_j$, a contradiction with the
  assumption that $c'_j\in I$.

  The argument for the case $c'_j\notin [d,c]$ is similar. Analogously
  to~\eqref{eq:1}, we find that $\size{(c,c'_j]_Y} \geq c'_j - c +
  1$. On the other hand, $(c,c'_j]_Y$ is independent and thus its size
  is at most $(c'_j-c)/2$, an improvement by $1/2$ coming from the
  fact that $c'_j\notin Y$ as $c'_j$ is a control element for $XY$. As
  a consequence, the resulting bound $c > c'_j + 1$ is even stronger
  than its analogue in the preceding case.

  A similar computation works for condition (ii).
\end{proof}

Throughout the following discussion, let $X$ be a $k$-element subset
of $W$. Let $i\leq m$. We say that $X$ is \emph{heavy} on $[d_i,c_i)$
if $\size{[d_i,c_i)_X} > c_i$. Furthermore, $X$ is \emph{light} or
\emph{balanced} on $[d_i,c_i)$ if $\size{[d_i,c_i)_X}$ is smaller
than or equal to $c_i$, respectively. These notions are also defined for
intervals $(d_i,c_i]$ or $[d_i,c_i]$ in the obvious way.

We say that $X$ is \emph{balanced} if it is balanced on every interval
$[d_i,c_i)$ and $(d_i,c_i]$, where $1\leq i \leq m$. The set $X$ is
\emph{regular} if it is balanced and contained in $A\cup B$. Note that
$A$ and $B$ are regular.

Let us say that $X$ is \emph{min-heavy} on $[d_i,c_i)$
($1\leq i \leq m$) if it is heavy on $[d_i,c_i)$ and not heavy on any
interval $[d_j,c_j)$ nor $(d_j,c_j]$ with $j < i$. Similarly, $X$ is
\emph{max-light} on $[d_i,c_i)$ if it is light on this interval and
not light on any $[d_j,c_j)$ nor $(d_j,c_j]$ with $j > i$. Being
min-heavy or max-light on the interval $(d_i,c_i]$ is defined in an
analogous manner.

A \emph{balanced pair} in $X$ is a pair $\opair{c_i}{d_i}$
($1\leq i \leq m$) such that $\Setx{c_i,d_i} \subseteq X$ and $X$ is
balanced on $[d_i,c_i)$ (and therefore also on $(d_i,c_i]$).

\begin{proposition}\label{p:irregular}
  Let $X$ be a $k$-element subset of $W$. If $X$ is not regular, then
  there exists $i\in[m]$ satisfying one of the following:
  \begin{enumerate}[label=(\alph*)]
  \item $\Setx{c_i,d_i}$ is a balanced pair in $X$,
  \item $X$ is min-heavy on $[d_i,c_i)$ or on $(d_i,c_i]$,
  \item $X$ is max-light on $[d_i,c_i)$ or on $(d_i,c_i]$.
  \end{enumerate}
\end{proposition}
\begin{proof}
  Suppose that $X$ is not regular. If there exists $j\in [m]$ such
  that $X$ is heavy or light on $[d_j,c_j)$ or $(d_j,c_j]$, then an
  index $i$ satisfying (b) or (c) can be obtained by making an
  appropriate extremal choice of $j$. We can thus assume that $X$ is
  balanced.

  Since $X$ is not regular, it contains an element from $C\cup D$ ---
  say, $d_\ell\in X$. (A symmetric argument works in the other case.)
  Being balanced, $X$ contains $c_\ell$ elements of $[d_\ell,c_\ell)$,
  and therefore $\size{(d_\ell,c_\ell)_X} = c_\ell-1$. Since
  $\size{(d_\ell,c_\ell]_X} = c_\ell$, we have $c_\ell\in X$. Thus,
  $\Setx{c_\ell,d_\ell}$ is a balanced pair in $X$.
\end{proof}

For convenience, we set $d_0 = c_1$, $c_0 = d_1$, $d_{m+1} = c_m$ and
$c_{m+1} = d_m$. For $i\in [m+1]$, we define
\begin{equation*}
  U_i = (d_i,d_{i-1})_W\cup (c_{i-1},c_i)_W.
\end{equation*}
Note that for each $i$, $U_i\subseteq A\cup B$ and $\size{U_i}$ is
even, namely
\begin{equation*}
  \size{U_i} =
  \begin{cases}
    2c_1 & \text{if $i = 1$,} \\
    2(k-c_m) & \text{if $i = m+1$,} \\
    2(c_i - c_{i-1}) & \text{otherwise.}
  \end{cases}
\end{equation*}

\begin{proposition}\label{p:half}
  Let $X\subseteq W$ and $i\in [m]$.
  \begin{enumerate}[label=(\roman*)]
  \item If $X$ is min-heavy on $[d_i,c_i)$ (respectively,
    $(d_i, c_i]$), then either $i > 1$ and $\Setx{c_{i-1},d_{i-1}}$ is
    a balanced pair in $X$, or $X$ contains more than half of the
    elements in the set $U_i\cup\Setx{d_i}$ (respectively,
    $U_i\cup\Setx{c_i}$).
  \item If $X$ is max-light on $[d_i,c_i)$ (respectively,
    $(d_i, c_i]$), then either $i < m$ and $\Setx{c_{i+1},d_{i+1}}$ is
    a balanced pair in $X$, or $X$ contains more than half of the
    elements in the set $U_{i+1}\cup\Setx{c_i}$ (respectively,
    $U_{i+1}\cup\Setx{d_i}$).
  \end{enumerate}
\end{proposition}
\begin{proof}
  We prove (i) only for the case of $X$ min-heavy on $[d_i,c_i)$ since
  the other case is completely analogous. For $i = 1$, the claim is
  trivially true since $X$ is heavy on
  $[d_1,c_1] = U_1\cup\Setx{d_1}$. Suppose then that $i > 1$.

  Since $X$ is heavy on $[d_i,c_i)$,
  $\size{[d_i,c_i)_X} \geq c_i + 1$. Let us assume that $X$ contains
  less than half of the elements of the (odd-sized) set
  $U_i\cup\Setx{d_i}$ --- that is, $\size{X\cap (U_i\cup\Setx{d_i})}
  \leq c_i - c_{i-1}$. Hence
  \begin{equation}\label{eq:heavy}
    \size{[d_{i-1},c_{i-1}]_X} \geq c_{i-1} + 1.
  \end{equation}
  On the other hand, $X$ is heavy on neither $[d_{i-1},c_{i-1})$ nor
  $(d_{i-1},c_{i-1}]$, so $\size{[d_{i-1},c_{i-1})_X} \leq c_{i-1}$ and
  $\size{(d_{i-1},c_{i-1}]_X} \leq c_{i-1}$. Comparing
  with~\eqref{eq:heavy}, we see that $c_{i-1},d_{i-1}\in
  X$. Furthermore, $\size{[d_{i-1},c_{i-1})_X} = c_{i-1}$, so
  $\Setx{c_{i-1},d_{i-1}}$ is a balanced pair in $X$.

  The proof of (ii) is similar and we only comment on the case of $X$
  max-light on $[d_i,c_i)$ and $i < m$. We have
  $\size{[d_i,c_i)_X} \leq c_i - 1$. If $X$ contains less than half of
  the elements in $U_{i+1}\cup\Setx{c_i}$, then
  $\size{(d_{i+1},c_{i+1})_X} \leq (c_i-1) + (c_{i+1}-c_i) =
  c_{i+1}-1$. However, $X$ is not light on $[d_{i+1},c_{i+1})$ nor on
  $(d_{i+1},c_{i+1}]$, so $c_{i+1},d_{i+1}\in X$ and
  $\size{[d_{i+1},c_{i+1})_X} = c_{i+1}$. It follows that
  $\Setx{c_{i+1},d_{i+1}}$ is a balanced pair in $X$.
\end{proof}

For $i\in [m+1]$, a set $X\subseteq W$ is \emph{skew at $d_i$} if $X$
contains the largest element of $(d_{i+1},d_i)_W$ and the second
smallest element of $(d_i,d_{i-1})_W$. (By Lemma~\ref{l:placement}(i),
each of the latter two sets contains at least two elements.)

Let $X\subseteq W$ be a vertex of $G$ and let $d\in [k,n]$. By
Lemma~\ref{l:depth}, there is at most one element $c\in [k-1]$ such
that $\size{[d,c]_X} = c$ and $c\notin X$. If such an element exists,
we call it the \emph{depth of $d$ in $X$} and define
$\delta(X,d) = c$; otherwise, we let $\delta(X,d) = k$.

This notion will only be used for vertices $X$ containing a
$W$-consecutive pair. For such a vertex, let $\opair s {s'}$ be a
$W$-consecutive pair in $X$ with $s$ as large as possible. (The choice
is not really essential, but we specify it to make the definition
unambiguous.) The \emph{depth $\delta(X)$ of $X$} is defined as
$\delta(X,s')$.

\begin{lemma}\label{l:depth-control}
  If $XY$ is an edge of $G$ with $X,Y\subseteq W$ and $X$ contains a
  $W$-consecutive pair, then $\delta(X)$ is one of the control
  elements for $XY$. In particular, $\delta(X) < k$ and
  $\delta(X)\notin X\cup Y$.
\end{lemma}
\begin{proof}
  Let $\opair s {s'}$ be a $W$-consecutive pair in $X$ with
  $\delta(X) = \delta(X,s')$. Let
  $\Set{\opair{c'_i}{d'_i}}{i\in [\ell]}$ be the set of control pairs
  of the standard $XY$-alternator. By Lemma~\ref{l:placement}(ii)
  (applied to the edge $XY$), we must have $d'_i = s'-1$ for some
  $i\in [\ell]$, in which case
  $\size{[s',c'_i)_X} = \size{[d'_i,c'_i)_X} = c'_i$. Furthermore,
  $c'_i\notin X$, so $c'_i$ is the depth of $s'$ in $X$, i.e.,
  $c'_i = \delta(X)$. The lemma follows.
\end{proof}

\begin{lemma}
  \label{l:depth-indep}
  Let $X,Y\subseteq W$ be vertices of $G$ such that
  $X\cap Y = \emptyset$, each of $X$ and $Y$ contains a
  $W$-consecutive pair, and $\delta(X) = \delta(Y)$. Then $X$ and $Y$
  are non-adjacent in $G$.
\end{lemma}
\begin{proof}
  Let $\delta = \delta(X)$. By Lemma~\ref{l:depth-control}, we may
  assume that $\delta \leq k-1$. Thus let $\opair s {s'}$ be a
  $W$-consecutive pair in $X$ with $\delta(X) = \delta(X,s')$, and
  similarly let $\opair t {t'}$ be such a pair in $Y$. By symmetry, we
  may assume that $t' < s'$, and therefore $t' < s$. By the definition
  of depth, $\size{[s',\delta)_X} = \size{[t',\delta)_Y} =
  \delta$. Since $t' < s$, we have $\size{[s,\delta)_Y} \leq \delta-1$
  while $\size{[s,\delta)_X} = \delta + 1$. Lemma~\ref{l:disbalance}
  implies that $X$ and $Y$ are non-adjacent.
\end{proof}

We are now ready to define a colouring of $G$ using the following set
of colours:
\begin{equation*}
  \Set{\col i}{i\in [n]\setminus (A\cup B)}\cup \Setx{\col 0}.
\end{equation*}
Since $\size A = \size B = k$, the total number of colours is
$n-2k+1$.

From this point on, we drop the assumption that $X\subseteq W$. A
vertex $X$ of $G$ is assigned a colour by the following rules, in the
stated order of precedence:
\begin{enumerate}[label=(R\arabic*)]
\item\label{rule:iness} If $X$ is inessential and therefore contains
  an element of $[n]\setminus W$, then it gets colour $\col j$, where
  $j$ is the least such element.
\item\label{rule:pair} If $X$ contains a balanced pair, then $X$ gets colour
  $\col{c_i}$, where $i\in [m]$ is least such that $\Setx{c_i,d_i}$ is
  a balanced pair in $X$.
\item\label{rule:ccw} If $X$ is min-heavy or max-light on $(d_i,c_i]$ for some
  $i\in [m]$, then $X$ gets colour $\col{c_i}$.
\item\label{rule:cw} If $X$ is min-heavy or max-light on $[d_i,c_i)$ for some
  $i\in [m]$, then $X$ gets colour $\col{d_i}$.
\item\label{rule:depth} If $X$ contains a $W$-consecutive pair and
  $\delta(X) = j$, then $X$ gets colour $\col j$ if $j\in
  [k-1]\setminus (A\cup B)$, and colour $\col 0$ otherwise (that is, if
  $j = k$ or $j \in [k-1]\cap (A\cup B)$).
\item\label{rule:skew} If $X$ is skew at $d_i$ for some $i\in [m]$,
  then $X$ gets colour $\col{d_i}$, where $i$ is least with this
  property.
\item\label{rule:zero} If none of the above applies, $X$ gets colour
  $\col{0}$.
\end{enumerate}

We will now show that each colour class of this colouring is an
independent set in $G$.

\begin{proposition}
  \label{p:colouring}
  Rules \ref{rule:iness}--\ref{rule:zero} determine a valid colouring
  of $G$.
\end{proposition}
\begin{proof}
  We will discuss each colour class in turn. If
  $j\in [n]\setminus (W\cup [k-1])$, then colour $\col j$ is only
  assigned by Rule~\ref{rule:iness}, namely to those vertices that
  contain element $j$. The colour class is therefore independent.

  \setcounter{claim}{0}
  \begin{claim}\label{cl:k}
    If $j\in [k-1]\setminus W$, then the colour class of $\col
    j$ is independent.
  \end{claim}
  \begin{claimproof}
    Colour $\col j$ may be assigned to $X$ by Rule~\ref{rule:iness}
    (if $j\in X$) or by Rule~\ref{rule:depth} (if $X$ contains a
    $W$-consecutive pair and $\delta(X) = j$). Suppose that vertices
    $X,Y$ both get colour $\col j$.

    Suppose that $j\in X$. If $j\in Y$, then $XY$ is not an edge. If
    $Y$ contains a $W$-consecutive pair and $\delta(Y) = j$, then
    $\delta(Y)\in X\cup Y$, so $XY$ is not an edge by
    Lemma~\ref{l:depth-control}.

    The last remaining case is that $X,Y$ both contain a
    $W$-consecutive pair and $\delta(X) = \delta(Y) = j$. In this
    case, $X$ and $Y$ are non-adjacent by Lemma~\ref{l:depth-indep}.
  \end{claimproof}

  At this point, it remains to consider all the colours $\col j$ with
  $j\in C\cup D$ and the colour $\col 0$. This is done in the
  following three claims.

  \begin{claim}\label{cl:c}
    For $i\in [m]$, the colour class of $\col{c_i}$ is independent.
  \end{claim}
  \begin{claimproof}
    Colour $\col{c_i}$ is assigned by Rules~\ref{rule:pair},
    \ref{rule:ccw} and \ref{rule:depth} to essential vertices $X$
    satisfying one of the following:
    \begin{itemize}
    \item $X$ contains a balanced pair $\Setx{c_i,d_i}$,
    \item $X$ contains no balanced pair and $X$ is min-heavy on $(d_i,c_i]$,
    \item $X$ contains no balanced pair and $X$ is max-light on $(d_i,c_i]$,
    \item $X$ contains a $W$-consecutive pair and $\delta(X) = c_i$.
    \end{itemize}

    Let $X$ and $Y$ be vertices of $G$ assigned colour $\col{c_i}$. We
    prove that $XY$ is not an edge of $G$. If both $X$ and $Y$ contain
    $\Setx{c_i,d_i}$, then they are intersecting and therefore
    non-adjacent in $G$. If both are min-heavy or both are max-light
    on $(d_i,c_i]$, they intersect by Proposition~\ref{p:half}.

    Suppose that $X$ is min-heavy on $(d_i,c_i]$. If $Y$ is max-light
    on $(d_i,c_i]$, then $\size{(d_i,c_i]_X} > c_i$ and
    $\size{(d_i,c_i]_Y} < c_i$, so Lemma~\ref{l:disbalance}(i) (with
    $c = c_i$ and $d = d_i+1$) shows that $X,Y$ are non-adjacent.

    If $\Setx{c_i,d_i}$ is a balanced pair in $Y$, then we may suppose
    that $c_i\notin X$. Thus, $\size{(d_i,c_i)_X} > c_i$. On the other
    hand, $\size{[d_i,c_i]_Y} = c_i + 1$, so
    $\size{(d_i,c_i)_Y} = c_i - 1$. Lemma~\ref{l:disbalance}(ii) (with
    $c = c_i$ and $d = d_i+1$) implies that $X,Y$ are non-adjacent.

    It remains to consider the case that one of $X$ and $Y$, say $X$,
    contains a $W$-consecutive pair. By the position of
    Rule~\ref{rule:depth}, it may be assumed that $X$ is regular; in
    particular, $\size{(d_i,c_i]_X} = c_i$. Let $\opair s {s'}$ be a
    $W$-consecutive pair contained in $X$ with $s$ as large as
    possible. Since $\delta(X) = c_i$, we have $s'\in (d_i,c_i]$ and
    $s\notin (d_i,c_i]$, so $s = d_i$.

    If $Y$ also contains a $W$-consecutive pair, then the same
    argument applies; therefore, $d_i\in X\cap Y$ and we are done. If $Y$ is
    min-heavy on $(d_i,c_i]$, then we may assume that $s'\notin Y$
    (otherwise $X$ and $Y$ intersect). Thus
    $\size{(s',c_i]_Y} = \size{(d_i,c_i]_Y} > c_i$, while
    $\size{(s',c_i]_X} = c_i-1$. Lemma~\ref{l:disbalance} implies that
    $X$ and $Y$ are non-adjacent. Finally, if $Y$ is max-light on
    $(d_i,c_i]$, then it may be assumed that $d_i\notin Y$ (otherwise,
    $X$ and $Y$ are intersecting), so $\size{[d_i,c_i]_Y} <
    c_i$. Since $\size{[d_i,c_i]_X} = c_i+1$, $X$ and $Y$ are
    non-adjacent by Lemma~\ref{l:disbalance}. This finishes the proof
    of Claim~\ref{cl:c}.
  \end{claimproof}

  \begin{claim}\label{cl:d}
    For $i\in [m]$, the colour class of $\col{d_i}$ is independent.
  \end{claim}
  \begin{claimproof}
    Note that since $d_i > k$ by the definition of $AB$-alternator,
    Rule~\ref{rule:depth} does not assign colour $\col{d_i}$. Thus,
    colour $\col{d_i}$ is only assigned by Rules~\ref{rule:cw} and
    \ref{rule:skew} to vertices $X$ satisfying one of the following:
    \begin{itemize}
    \item $X$ contains no balanced pair and is min-heavy on
      $[d_i,c_i)$,
    \item $X$ contains no balanced pair and is max-light on
      $[d_i,c_i)$,
    \item $X$ is regular and skew at $d_i$.
    \end{itemize}

    Suppose that $X$, $Y$ are disjoint vertices of $G$ assigned colour
    $\col{d_i}$. We prove that $X$ and $Y$ are non-adjacent in $G$.

    Suppose first that both $X$ and $Y$ are min-heavy or max-light on
    $[d_i,c_i)$. Let $X$ be min-heavy on $[d_i,c_i)$. By
    Proposition~\ref{p:half}(i), $Y$ is not min-heavy on $[d_i,c_i)$,
    for otherwise $X$ and $Y$ would intersect. Thus we may assume
    that $Y$ is max-light on $[d_i,c_i)$, but then
    $\size{[d_i,c_i)_Y} < c_i$ and $\size{[d_i,c_i)_X} > c_i$, so $XY$
    is not an edge by Lemma~\ref{l:disbalance}. A symmetric argument
    applies in case $X$ is max-light on $[d_i,c_i)$.

    Assume thus that $X$ is regular and skew at $d_i$. We will
    also assume that $i > 1$, the $i=1$ case being analogous. By
    the position of Rule~\ref{rule:depth}, $X$ contains no $W$-consecutive
    pair. Consider the set $S = U_i\cup \Setx{d_i}$ and recall that
    $\size S = 2(c_i-c_{i-1})+1$. Since $X$ is regular,
    $\size{X\cap S} = c_i-c_{i-1}$.

    We now distinguish several cases according to the type of $Y$. If $Y$
    is skew at $d_i$, then $X$ and $Y$ both contain the second smallest
    element of $(d_i,d_{i-1})_W$, a contradiction.
    
    Assume next that $Y$ is min-heavy on $[d_i,c_i)$. By
    Proposition~\ref{p:half}, $\size{Y\cap S} = c_i-c_{i-1}+1$, so the
    sets $X\cap S$ and $Y\cap S$ partition $S$. Now $d_i\notin X$
    (since $X$ is regular), so $d_i\in Y$. Since $C\cup D$ is the
    standard $AB$-alternator, the
    smallest element of $(d_i,d_{i-1})_W$ is $d_i+1$. We have
    $d_i+1\notin X$ as $X$ is skew at $d_i$ and contains no
    $W$-consecutive pair. It follows that
    $\Setx{d_i,d_i+1}\subseteq Y$, a contradiction with the
    independence of $Y$ in $C_n$.

    It remains to consider the case that $Y$ is max-light on
    $[d_i,c_i)$. Let $d^-$ be the last element of
    $(d_{i+1},d_i)_W$. Since $d^-\in X$, we have $d^-\notin
    Y$. Furthermore, $\size{[d^-,c_i)_X} = c_i + 1$ as $X$ is regular
    on $[d_i,c_i)$, while $\size{[d^-,c_i)_Y} <
    c_i$. Lemma~\ref{l:disbalance} implies that $X$ and $Y$ are
    non-adjacent in $G$. The proof of Claim~\ref{cl:d} is complete.
  \end{claimproof}

  \begin{claim}\label{cl:0}
    The colour class of $\col 0$ is independent.
  \end{claim}
  \begin{claimproof}
    Colour $\col 0$ is assigned by Rule~\ref{rule:depth} to vertices
    $X$ containing a $W$-consecutive pair and having depth in the set
    $[1,k-1]_{A\cup B}\cup\Setx{k}$, and by Rule~\ref{rule:zero} to
    vertices satisfying none of the conditions in
    Rules~\ref{rule:iness}--\ref{rule:skew}. Let us say that $X$ is of
    type~\ref{rule:depth} or~\ref{rule:zero} accordingly. All of these
    vertices are regular (hence subsets of $W$) by
    Proposition~\ref{p:irregular}; additionally, type~\ref{rule:zero}
    vertices contain no $W$-consecutive pair, and are not skew at any
    $d_i$ ($i\in [m]$).

    Suppose that $X$ is a vertex of type~\ref{rule:depth} and $Y$ is
    any vertex coloured $\col 0$ with $X\cap Y = \emptyset$. If
    $\delta(X) = k$ then $X$ is not adjacent to $Y$ by
    Lemma~\ref{l:depth-control}. We may thus assume that
    $\delta(X) \in [1,k-1]_{A\cup B}$. Since $X$ and $Y$ are regular,
    we have $X\cup Y = A\cup B$ and therefore $\delta(X)\in X\cup
    Y$. Lemma~\ref{l:depth-control} implies that $XY$ is not an edge.

    This leaves us with the case that $X,Y$ are vertices of
    type~\ref{rule:zero}. We intend to show that
    $\Setx{X,Y} = \Setx{A,B}$. The set $(d_1,c_1)_W$ consists of
    $2c_1$ elements and each of $X$ and $Y$ contains $c_1$ of
    them. Since none of $X$ and $Y$ contains a $W$-consecutive pair,
    we may assume by symmetry that $(d_1,c_1)_X = (d_1,c_1)_A$. We
    prove by induction that for each $i\in [m]$,
    $(d_i,d_{i-1})_X = (d_i,d_{i-1})_A$; recalling the convention that
    $c_1 = d_0$, the base case is established. Consider $i\geq 2$. Let
    $d^-$ be the largest element of $(d_{i+1},d_i)_W$, and let
    $d', d''$ be the smallest and second smallest element of
    $(d_i,d_{i-1})_W$, respectively. (Recall that each of these sets
    has size at least $2$ by Lemma~\ref{l:placement}(i).) By the
    induction hypothesis and the assumption that $X$ is not skew at
    $d_i$, we have
    \begin{equation*}
      d^-\in X \iff d''\notin X \iff d''\notin A\iff d'\in A\iff d^-\in A.
    \end{equation*}
    Since each element of $(d_{i+1},d_i)_W$ belongs to $X$ or $Y$, and
    since $X$ and $Y$ contain no $W$-consecutive pair, this implies
    that $(d_{i+1},d_i)_X = (d_{i+1},d_i)_A$ as required.

    The above implies that $[k,n]_X = [k,n]_A$. Since $Y$ is regular,
    $[k,n]_Y = [k,n]_B$.

    To prove that $X=A$, it remains to show that
    $(c_i,c_{i+1})_X = (c_i,c_{i+1})_A$ for each $i\in [m-1]$. By
    Lemma~\ref{l:placement}(ii), $d_i$ is contained in a control pair
    for the edge $XY$. Since $X$ and $Y$ are regular, it follows that
    the control pair must be $\opair{c_i}{d_i}$. Let $d$ be the
    largest element of $(d_{i+1},d_i)_W$ and let $c$ be the smallest
    element of $(c_i,c_{i+1})_W$. By the definition of the edge set of
    $G$,
    \begin{equation*}
      c\in X \iff d\notin X\iff d\notin A\iff c\in A.
    \end{equation*}
    Since $X$ and $Y$ contain no $W$-consecutive pair and cover
    $(c_i,c_{i+1})_W$, this determines $(c_i,c_{i+1})_X$ and shows
    that this set equals $(c_i,c_{i+1})_A$. The proof that $X=A$ (and
    $Y=B$) is complete. Since the edge $AB$ does not exist in $G$,
    this finishes the proof of the claim.
  \end{claimproof}
  The proof of Proposition~\ref{p:colouring} is complete.
\end{proof}

\bibliographystyle{hplain}
\bibliography{XG}

\end{document}